\documentclass[12pt]{article}
\usepackage{amsmath, amsthm, amscd, amsfonts, amssymb, graphicx}
\usepackage{titles}
\usepackage{algorithm}
\usepackage{algpseudocode}
\usepackage{subfigure} 
\usepackage{caption}
\usepackage{graphics,graphicx}
\usepackage{amsmath, amssymb, amsthm}
\usepackage{mathrsfs} 
\usepackage{enumitem}

\usepackage{indentfirst}

\numberwithin{equation}{section}
\usepackage{epsfig}
\usepackage{color}
\usepackage{fancyhdr}
\usepackage{titlesec}
\usepackage{tikz}
\usepackage{afterpage}
\usetikzlibrary{arrows}
\usetikzlibrary{decorations.markings}
\tikzstyle{block}=[draw opacity=0.7,line width=1.4cm]
\tikzset{
	big black arrow/.style={
		decoration={markings,mark=at position 1 with {\arrow[scale=2.5,black]{>}}},
		postaction={decorate},
		shorten >=0.4pt},
	line/.style={draw, ->}}

\usepackage{eufrak}
\setbox0=\hbox{$+$}
\newdimen\plusheight
\plusheight=\ht0
\def\+{\;\lower\plusheight\hbox{$+$}\;}

\setbox0=\hbox{$-$}
\newdimen\minusheight
\minusheight=\ht0
\def\-{\;\lower\minusheight\hbox{$-$}\;}

\setbox0=\hbox{$\cdots$}
\newdimen\cdotsheight
\cdotsheight=\plusheight
\def\cds{\lower\cdotsheight\hbox{$\cdots$}}
\newtheorem{theorem}{Theorem}[section]

\newtheorem{corollary}[theorem]{Corollary}

\theoremstyle{definition}

\theoremstyle{remark}

\numberwithin{equation}{section}

\allowdisplaybreaks

\date{}

\begin{document}
\begin{center}\Large{\textbf{Some new congruences and identities  for $SOME(n)$, $DSOME(n)$, $\overline{SOME}(n)$ functions and analogues}}\end{center}\vskip.5cm

\linespread{1.2}
\begin{center}
	{\bf Gaurab Bardhan$^1$ and Nipen Saikia$^{2, \ast}$}\\
	$^1$Department of Mathematics, Tyagbir Hem Baruah College,\\ Jamugurihat, Sonitpur, Assam, India.\\
	E. Mail: gaurabbardhan561@gmail.com
	\vskip2mm
	$^2$Department of Mathematics, Rajiv Gandhi
	University,\\ Rono Hills, Doimukh, Arunachal Pradesh, India.\\
	E. Mail(s): nipennak@yahoo.com\\
	$^\ast$\textit{Corresponding author}.\end{center}
\begin{abstract}
Andrews and Dastidar (\textit{Ramanujan J. 69, Article Num- ber 26, (2026)} ) introduced the $SOME(n)$ and $DSOME(n)$ functions that calculate the sum of all odd parts minus the sum of all even parts of ordinary partitions and distinct partitions,  respectively of a positive integer $n$, and proved their generating functions and some congruences modulo 4 and 5.  Recently,  Gireesh and Hemanthkumar introduced an overpartition analogue of $SOME(n)$ function, denoted by $\overline{SOME}(n)$ and proved some congruences modulo 3, 5 and powers of 2. In this paper,   we prove  some new identities and congruences for $SOME(n)$, $DSOME(n)$, and $\overline{SOME}(n)$ functions, including monotonicity results. We also define a general analogue of $SOME(n)$ function, denoted by $S_{\mathcal P}(n)$, which calculates the sum of all odd parts minus the sum of all even parts in  any arbitrary  family of partitions $\mathcal P(n)$ of a positive integer $n$, and prove some divisibility properties. Additionally, we define a colour partition analogue of $SOME(n)$ function and  prove divisibility properties.
\end{abstract}\vskip3mm
\noindent  {\bf Keywords and phrases:} Integer partitions; $SOME$ function and analogues; Identities; Congruences·
\vskip 2mm
\noindent  {\bf Mathematical Subject Classification:} Primary-11P81, 11P83; Secondary-05A15, 05A17.
\section{Introduction}
For any positive integer $n$ and any complex numbers $\alpha$ and $q$ with $|q|<1$, define the standard $q$-series notation as
$$
(\alpha;q)_{0}=1,\quad  (\alpha;q)_{n}=\prod_{j=0}^{n-1}(1-\alpha q^{j}),\quad(\alpha;q)_{\infty}=\prod_{j=0}^{\infty}(1-\alpha q^{j}).
$$
Throughout this paper, we use the notation
$$
g_t:=(q^t;q^t)_\infty
$$
for any positive integer $t$.

A non-increasing finite sequence of positive integers $\beta_1\geq \beta_2 \geq \cdots \geq \beta_k>0$ is said to be a partition of a positive integer $n$ if $n=\sum_{i=1}^{k}\beta_i$.
The integers $\beta_i$ are called the parts of the partition. If $p(n)$ denotes the number of partitions of $n$, then its  generating function \cite{euler1748introductio} is given by
\begin{equation}\label{ie4}
\sum_{n=0}^{\infty}p(n)q^n=\dfrac{1}{(q;q)_\infty}=\dfrac{1}{g_1};\qquad p(0)=1.
\end{equation}
Also, if  $p_d(n)$ denotes the number of partitions of $n$ into distinct parts with $p_d(0)=1$, then its generating function is given by
 \begin{equation}\label{p_d}
 \sum_{n=0}^{\infty}p_d(
 n)q^n= (-q;q)_\infty.
\end{equation}
An overpartition of a positive integer $n$ is a partition of $n$  in which the first occurrence (equivalently, the final occurrence) of a part may be overlined.  If the number of overpartitions of $n$ is denoted by $\overline{p}(n)$, then its generating function \cite{cor} is given by 
$$
	\sum_{n=0}^{\infty}\overline{p}(n)=\dfrac{(-q;q)_\infty}{(q;q)_\infty}.
$$
Recently, Andrews and Ghosh Dastidar \cite{das} introduced two partition related functions, $SOME(n)$ and $DSOME(n)$, where  $SOME(n)$ calculates the sum of all odd parts in all partitions of $n$ minus the sum of all even parts in all partitions of $n$, and  $DSOME(n)$  calculates the sum of all odd parts  minus the sum of all even parts  in all the partitions of $n$ with distinct parts. They gave the generating functions of $SOME(n)$ and $DSOME(n)$ as, for 
\begin{align}
\sum_{n=0}^{\infty}SOME(n)q^n
&\label{i6}=
\dfrac{1}{(q;q)_\infty}\sum_{m=1}^{\infty}\dfrac{q^m}{(1+q^m)^2},\\
&\label{Tr}=\dfrac{1}{(q^2;q^2)_\infty^2}\sum_{m=0}^{\infty}T_mq^{T_m},
\end{align}
where $T_m=\dfrac{m(m+1)}{2},$ and
\begin{align}
\sum_{n=0}^{\infty}DSOME(n)q^n
&=
(-q;q)_\infty\sum_{m=1}^{\infty}\dfrac{ (-1)^{m-1}q^m}{(1+q^m)^2}.
\label{i7}
\end{align}
Moreover,  they proved the following interesting congruences satisfied by $SOME(n)$ and $DSOME(n)$:
$$
SOME(5n+2)\equiv0\pmod 5,\quad
SOME(5n+4)\equiv0\pmod 5,
$$
$$
SOME(4n)\equiv0\pmod4,\quad
DSOME(4n)\equiv0\pmod4
$$  and  also conjectured that, if $24n \equiv1 \pmod{5^\alpha}$, then $$SOME(n) \equiv0 \pmod{5^\alpha}.$$
Baruah and Gogoi \cite{NDB} expressed the generating function of $DSOME(n)$ in the closed form as
$$
\sum_{n=0}^{\infty}DSOME(n)q^n
=
\dfrac{1}{8}
\left(
\dfrac{g_2}{g_1}
-
\dfrac{g_1^7}{g_2^3}
\right)
$$
and proved some new congruences  modulo $4$ and $8$ for $DSOME(n)$.
Gireesh and Hemanthkumar \cite{gh} studied on an overpartition analogue of the function $SOME(n)$, denoted by $\overline{SOME}(n)$,  which is defined as the sum of all odd parts minus the sum of all even parts taken over all overpartitions of $n$.
They obtained the generating function of $\overline{SOME}(n)$ as
\begin{equation}\label{ie13}
\sum_{n=1}^\infty \overline{SOME}(n)q^n = 2 \dfrac{(-q; q)_\infty}{(q; q)_\infty} \sum_{m=1}^\infty \dfrac{q^{2m-1}}{(1 + q^{2m-1})^2},
\end{equation}
and proved several congruences and other arithmetic properties of $\overline{SOME}(n)$.

Motivated by above work, in this paper,  we prove  some new identities and congruences for $SOME(n)$, $DSOME(n)$, and $\overline{SOME}(n)$ functions, including monotonicity results. We also define a general analogue of $SOME(n)$ function, denoted by $S_{\mathcal P}(n)$, which calculates the sum of all odd parts minus the sum of all even parts in  any arbitrary  family of partitions $\mathcal P(n)$ of a positive integer $n$, and prove some divisibility properties. Additionally, we define a colour partition analogue of $SOME(n)$ function and  prove divisibility properties.

The layout of the paper is as follows:  In Sect. 2, we state some preliminary results that will be useful for establishing our main results. In Sect. 3, we prove  some new identities and congruences for $SOME(n)$, $DSOME(n)$, and $\overline{SOME}(n)$ functions, including monotonicity results.  In  Sect. 4,  we define the function $S_{\mathcal P}(n)$ and prove its divisibility properties.  Finally, in Sect. 5, we  define a colour partition analogue of $SOME(n)$ function and prove some divisibility properties.

\section{Preliminaries}
Recall, Ramanujan's general theta function $f(\alpha, \beta)$ \cite[p. 34, (18.1)]{BBC} is defined by
$$
   f(\alpha, \beta) = \sum_{m=-\infty}^{\infty} \alpha^{m(m+1)/2} \beta^{m(m-1)/2}, \qquad |\alpha\beta| < 1.
$$
Three important special cases for $f(\alpha, \beta)$ are the functions $\phi(q), \psi(q)$, and $f(-q)$ \cite[p. 35, Entry 18]{BBC} ,which are defined as
\begin{equation}\label{phi}
    \phi(q) := f(q, q) = \sum_{m=-\infty}^{\infty} q^{m^2} = \dfrac{g_2^5}{g_1^2g_4^2},
\end{equation}
\begin{equation}\label{psi}
    \psi(q) := f(q, q^3) = \sum_{m=-\infty}^{\infty} q^{m(m+1)/2} = \dfrac{g_2^2}{g_1},
\end{equation}
and 
\begin{equation}\label{f}
    f(-q) := f(-q, -q^2) = \sum_{m=-\infty}^{\infty} (-1)^m q^{m(3m-1)/2} = g_1.
\end{equation}
Let $\sigma(n)$ defined by 
$$
\sigma(n)=\sum_{d\mid n}d
$$
be the ordinary sum of divisors function. 
Then the Lambert series for  $\sigma(n)$ is given by 
	$$
	\sum_{n=1}^\infty \dfrac{n q^{n}}{1 - q^{n}}=\sum_{n=1}^\infty \sigma(n)q^n.	
    $$
Also, if $r_k(n)$ denotes the the number of representations of a any positive integer $n$ by the sum of $k$ squares of any integers, then from \cite[Eq. (3.2.1), (3.3.1)]{BBC}, we have 
\begin{equation}\label{r_2}
    r_2(n)=4\sum_{\substack{d|n\\d odd}}(-1)^{(d-1)/2}.
\end{equation}
and 
\begin{equation}\label{r_4}
r_4(n)=8\sum_{\substack{d|n\\4\nmid d }}d.
\end{equation}

\section{Identities and congruences for $SOME(n),$ $DSOME(n),$ and $\overline{SOME}(n)$}
\begin{theorem}
For any integer $n\geq 1$, we have 
$$
SOME(n)+4\sum_{i=1}^{\lfloor n/2\rfloor}p(n-2i)\sigma(i)=np(n), 
$$where $\lfloor \cdot\rfloor$ denotes the floor function.
\end{theorem}
\begin{proof} We note that,
\begin{align}
\sum_{m=1}^{\infty}\dfrac{q^m}{(1+q^m)^2}
&=
\sum_{m=1}^{\infty}\sum_{d=1}^{\infty}(-1)^{d-1}d q^{md} \notag\\
&=
\sum_{j=1}^{\infty}\sum_{d\mid j}(-1)^{d-1}d q^j\notag\\
&=
\sum_{j=1}^{\infty}\left(\sum_{d\mid j}d
-
2\sum_{\substack{d\mid j\\ d\ \mathrm{even}}}d\right)q^j\notag\\
&=
\sum_{j=1}^{\infty}\sigma(j)q^j
-
2\sum_{j=1}^{\infty}
\sum_{\substack{d\mid j\\ d\ \mathrm{even}}}d q^j\notag\\
&=\label{mt2e4}
\sum_{j=1}^{\infty}\sigma(j)q^j
-
2\sum_{j=1}^{\infty}
\sum_{\substack{d\mid j\\ d\ \mathrm{even}}}d q^j
\end{align}
We note that, if $d$ is even, say $d=2e$, where $e$ is a positive integer, then $d\mid j$ implies that $j$
is even, say $j=2i$, where $i$ is a positive integer and $e\mid i$. Therefore, for $j=2i, $
\begin{equation}\label{eq:even-divisor-sum}
\sum_{\substack{d\mid j\\ d\ \mathrm{even}}}d
=
2\sum_{e\mid i}e
=
2\sigma(i),
\end{equation}
Using \eqref{eq:even-divisor-sum} in  \eqref{mt2e4}, we obtain
\begin{equation}\label{eq:key-inner-series}
\sum_{m=1}^{\infty}\dfrac{q^m}{(1+q^m)^2}
=
\sum_{j=1}^{\infty}\sigma(j)q^j
-
4\sum_{i=1}^{\infty}\sigma(i)q^{2i}.
\end{equation}
Employing  \eqref{eq:key-inner-series} in \eqref{i6}, we obtain
\begin{equation}\label{mt2e5}
\sum_{n=1}^{\infty}SOME(n)q^n
=
\dfrac{1}{(q;q)_\infty}
\sum_{j=1}^{\infty}\sigma(j)q^j
-
4\dfrac{1}{(q;q)_\infty}
\sum_{i=1}^{\infty}\sigma(i)q^{2i}.
\end{equation}
Employing \eqref{ie4} in \eqref{mt2e5}, we obtain 
\begin{equation}\label{mt2e6}
    \sum_{n=1}^{\infty}SOME(n)q^n
=
\dfrac{1}{(q;q)_\infty}
\sum_{j=1}^{\infty}\sigma(j)q^j
-
4\dfrac{1}{(q;q)_\infty}
\sum_{i=1}^{\infty}\sigma(i)q^{2i}.
\end{equation}
Now by logarthmic differentiation of \eqref{ie4}, it is easy to see that,
\begin{equation}\label{mt2e7}
   \sum_{m=0}^{\infty}np(n)q^n= \dfrac{1}{(q;q)_\infty}
\sum_{j=1}^{\infty}\sigma(j)q^j.
\end{equation}
Employing  \eqref{ie4} and \eqref{mt2e7} in \eqref{mt2e6}, we obtain 
\begin{equation}\label{mt2e6i}
	\sum_{n=1}^{\infty}SOME(n)q^n
	=
	\sum_{m=0}^{\infty}np(n)q^n
	-
	4\left(\sum_{n=0}^{\infty}p(n)q^n\right)
	\sum_{i=1}^{\infty}\sigma(i)q^{2i}.
\end{equation}
Equating the coefficients of $q^n$ in both sides of \eqref{mt2e6i}, we arrive at the desired result. 
\end{proof}
\begin{theorem}
For any integer $n\geq1$,  we have
$$
\overline{SOME}(n)
\equiv
\begin{cases}
2 \pmod 4, & \text{if } n \text{ is an odd perfect square},\\
0 \pmod 4, & \text{otherwise}.
\end{cases}    
$$
\end{theorem}
\begin{proof}
From \cite[Eq. 38]{gh}, we note that
\begin{equation}\label{mt3e1}
    \sum_{n=1}^{\infty}\overline{SOME}(n)q^n
=
\dfrac{2}{\phi(-q^2)^2}
\sum_{k=1}^{\infty}k^2q^{k^2}.
\end{equation}
Replacing $q$ by $-q^2$ in \eqref{phi} and then employing it in \eqref{mt3e1}, we obtain
\begin{equation}\label{mt3e2}
    \left(
\sum_{a,b\in\mathbb Z}
(-1)^{a+b}q^{2(a^2+b^2)}
\right)
\left(
\sum_{n=1}^{\infty}\overline{SOME}(n)q^n
\right)
=
2\sum_{k=1}^{\infty}k^2q^{k^2}.
\end{equation}
Comparing the coefficients of like powers of $q$ on both sides of \eqref{mt3e2}, we obtain
\begin{equation}\label{mt3e3}
  \overline{SOME}(n)
+
\sum_{j=1}^{\lfloor n/2\rfloor}
(-1)^j r_2(j)\overline{SOME}(n-2j)
=
\begin{cases}
2n, & \text{if } n \text{ is a perfect square},\\[4pt]
0, & \text{otherwise}.
\end{cases}  
\end{equation}
Employing \eqref{r_2} in \eqref{mt3e3}, we arrive at the  desired result. 
\end{proof}

\begin{theorem}
For any integer $n\ge 1$, we have
$$
\overline{SOME}(n)
\equiv
\begin{cases}
2 \pmod 8, & \text{if } n \text{ is an odd perfect square},\\[4pt]
0 \pmod 8, & \text{otherwise}.
\end{cases}
$$
\end{theorem}

\begin{proof}
Multipliying both sides of \eqref{mt3e1} by $\phi(-q^2)^4$, we obtain
\begin{equation}\label{mt4e1}
    \phi(-q^2)^4
\sum_{n=1}^{\infty}\overline{SOME}(n)q^n
=
2\phi(-q^2)^2
\sum_{k=1}^{\infty}k^2q^{k^2}.
\end{equation}
Replacing $q$ by $-q^2$ in \eqref{phi} and then employing in \eqref{mt4e1}, we obtain
\begin{equation}\label{mt4e2}
\left(\sum_{j=0}^{\infty}(-1)^j r_4(j)q^{2j}\right)\left(\sum_{n=1}^{\infty}\overline{SOME}(n)q^n\right)=2\left(\sum_{j=0}^{\infty}(-1)^j r_2(j)q^{2j}\right)\left(\sum_{k=1}^{\infty}k^2q^{k^2}\right).
\end{equation}
Comparing the coefficients of the like powers of $q$ on both sides of \eqref{mt4e2}, we obtain
\begin{equation}\label{mt4e3}
    \sum_{j=0}^{\lfloor n/2 \rfloor}(-1)^jr_4(j)\overline{SOME}(n-2j)=2\sum_{\substack{k\geq1 \\k^2\leq n\\ n\equiv k^2\pmod2}}k^2(-1)^{(n-k^2)/2}r_2\left(\dfrac{n-k^2}{2}\right).
\end{equation}
Employing \eqref{r_4} in \eqref{mt4e3}, we obtain
\begin{align}
    \overline{SOME}(n)&=2\sum_{\substack{k>1 \\k^2\leq n\\ n\equiv k^2\pmod2}}k^2(-1)^{(n-k^2)/2}r_2\left(\dfrac{n-k^2}{2}\right)\notag\\
    &\label{mt4e3t}\quad-8\sum_{j=1}^{\lfloor n/2\rfloor}\left(-1\right)^j\left(\sum_{d|j}d\right)\overline{SOME}(n-2j).
\end{align}
Employing \eqref{r_2} in \eqref{mt4e3t} and simplifying, we arrive at the desired result.
\end{proof}
\begin{theorem}
For any integer $\alpha\geq 0$ and  $n\ge 1$, we have
$$
\overline{SOME}(n)
\equiv
0
\pmod{2^{\alpha+2}}
$$
if and only if
$$
\sum_{\substack{k\geq 1,\ m\geq 0\\ k^2+2m=n}}(-1)^m k^2 r_{2^{\alpha+1}-2}(m)\equiv0\pmod{2^{\alpha+1}},
$$
where $r_t(m)$ denotes the number of representations of $m$ as a sum of $t$
squares with the convention that  $r_0(0)=1$ and $r_0(m)=0$ for $m>0$.
\end{theorem}

\begin{proof}
Multiplying both sides of \eqref{mt3e1} by $\phi(-q^2)^{2^{\alpha+1}},$ we obtain
\begin{equation}\label{mt5e1}
\phi(-q^2)^{2^{\alpha+1}}
\sum_{n=1}^{\infty}\overline{SOME}(n)q^n
=
2\phi(-q^2)^{2^{\alpha+1}-2}
\sum_{k=1}^{\infty}k^2q^{k^2}.    
\end{equation}
Replacing $q$ by $-q^2$ in \eqref{phi} and then employing in \eqref{mt5e1}, we obtain
\begin{equation}\label{mt5e2}
\phi(-q^2)^{2^{\alpha+1}}\sum_{n=1}^{\infty}\overline{SOME}(n)q^n
=
\left(2\sum_{m=0}^{\infty}(-1)^m r_{2s-2}(m)q^{2m}\right)
\left(\sum_{k=1}^{\infty}k^2q^{k^2}\right).     
\end{equation}
Using Binomial theorem, it is easily  seen that 
\begin{equation}\label{mt5e3}
   \phi(-q^2)^{2^{\alpha+1}}=1+\sum_{i=1}^{2^{\alpha+1}}2^i\binom{2^{\alpha+1}}{i}\left(\sum_{j=1}^{\infty}6(-1)^j q^{2j^2}\right)^{i}\equiv1\pmod{2^{\alpha+2}}. 
\end{equation}
Employing \eqref{mt5e3} in \eqref{mt5e2} and comparing the coefficients of like powers of $q$, we obtain
\begin{equation}\label{mt5e4}
    \overline{SOME}(n)\equiv
2
\sum_{\substack{k\geq 1,\ m\geq 0\\ k^2+2m=n}}
(-1)^m k^2 r_{2^{\alpha+1}-2}(m)\pmod{2^{\alpha+2}}.
\end{equation}
Now, the desired result follows easily from \eqref{mt5e4}.
\end{proof}
It is useful to note that, any positive integer $t$ can be expressed as 
$$
t=2^{\nu_2(t)}u,
$$
where $u$ is odd and $\nu_2(t)$ is a non-negative integer.
Also, if  $t$ is even, then $\nu_2(t)\geq1,$ and any divisor of $t$ is of the form $2^ke,$ where $0\leq k\leq \nu_2(j),$ and $e|u.$ Using this basic concept we derive the following theorem.\\
\begin{theorem}
	For any integer  $n\geq 1$, we have 
	$$
		DSOME(n)
		=
		\sum_{\substack{1\leq j\leq n\\ j\ \mathrm{odd}}}
		p_d(n-j)\sigma(j)
		-
		3\sum_{\substack{1\leq j\leq n\\ j\ \mathrm{even}}}
		p_d(n-j)
		\sigma\!\left(\dfrac{j}{2^{\nu_2(j)}}\right).    
	$$
	where $\nu_2(j)$ is a positive integer.
\end{theorem}
\begin{proof}
	Employing \eqref{p_d} in \eqref{i7}, we obtain
	\begin{equation}\label{mTe2}
		\sum_{n=0}^{\infty}DSOME(n)q^n
		=
		\left(\sum_{r=0}^{\infty}p_d(r)q^r\right)
		\left(
		\sum_{m=1}^{\infty}
		\dfrac{(-1)^{m-1}q^m}{(1+q^m)^2}
		\right). 
	\end{equation}Simplifying the second sum in the right hand side of \eqref{mTe2}, we see that
	\begin{align}
		\sum_{m=1}^{\infty}
		\dfrac{(-1)^{m-1}q^m}{(1+q^m)^2}=&\sum_{m=1}^{\infty}\sum_{d=1}^{\infty}
		(-1)^{m+d}d q^{md}\notag\\
		=&\sum_{j=1}^{\infty}\sum_{d|j}
		(-1)^{d+j/d} ~d q^{j}\notag\\
		=&\sum_{\substack{j=1\\ j\ \mathrm{odd}}}^{\infty}\sum_{d|j}
		(-1)^{d+j/d}~d q^{j}+\sum_{\substack{j=1\\ j\ \mathrm{even}}}^{\infty}\sum_{d|j}
		(-1)^{d+j/d}~d q^{j}\notag\\
		=&\sum_{\substack{j=1\\ j\ \mathrm{odd}}}^{\infty}\sum_{d\mid j}dq^j+\sum_{\substack{j=1\\ j\ \mathrm{even}}}^{\infty}\sum_{e\mid u}\sum_{k=0}^{\nu_2(j)}
		(-1)^{2^k e+2^{\nu_2(j)-k}u/e}2^k eq^j\notag\\
		=&\sum_{\substack{j=1\\ j\ \mathrm{odd}}}^{\infty}\sigma(j)q^j+\sum_{\substack{j=1\\ j\ \mathrm{even}}}^{\infty}\sum_{e\mid u}\left(-e+\sum_{k=1}^{\nu_2(j)-1}2^k e-2^{\nu_2(j)} e\right)q^j\notag\\
		=&\sum_{\substack{j=1\\ j\ \mathrm{odd}}}^{\infty}\sigma(j)q^j-3\sum_{\substack{j=1\\ j\ \mathrm{even}}}^{\infty}\sum_{e\mid u}eq^j\notag\\
		=&\sum_{\substack{j=1\\ j\ \mathrm{odd}}}^{\infty}\sigma(j)q^j-3\sum_{\substack{j=1\\ j\ \mathrm{even}}}^{\infty}\sigma(u)q^j\notag\\
		=&\label{mTe3}\sum_{\substack{j=1\\ j\ \mathrm{odd}}}^{\infty}\sigma(j)q^j-3\sum_{\substack{j=1\\ j\ \mathrm{even}}}^{\infty}\sigma\left(\dfrac{j}{2^{\nu_2(j)}}\right)q^j
	\end{align}
	Employing \eqref{mTe3} in \eqref{mTe2}, and then comparing the like powers of $q$ proves the theorem.
\end{proof}

\begin{theorem}
For integer $n\ge 1$, we have
$$
DSOME(n)
=
\sum_{\substack{j\in\mathbb Z\\ j(3j-1)\leq n}}
(-1)^j
SOME\bigl(n-j(3j-1)\bigr)
-
2
\sum_{\substack{j\geq 0\\ j(j+1)/2\leq n\\ n\equiv j(j+1)/2\pmod 2}}
SOME\left(\dfrac{n-j(j+1)/2}{2}\right).    
$$
\end{theorem}
\begin{proof}
It is easy to see that 
\begin{equation}\label{mt6e2}
    \sum_{m=1}^{\infty}
\dfrac{(-1)^{m-1}q^m}{(1+q^m)^2}
=
\sum_{m=1}^{\infty}\dfrac{q^m}{(1+q^m)^2}
-
2\sum_{m=1}^{\infty}\dfrac{q^{2m}}{(1+q^{2m})^2}.
\end{equation}
Invoking \eqref{i6} and \eqref{i7} in \eqref{mt6e2}, we obtain
\begin{equation}\label{mt6e3}
\sum_{n=1}^{\infty}DSOME(n)q^n
=
(-q;q)_\infty
\left(
(q;q)_\infty
\sum_{n=1}^{\infty}SOME(n)q^n
-
2(q^2;q^2)_\infty
\sum_{n=1}^{\infty}SOME(n)q^{2n}
\right).
\end{equation}
Simplifying  \eqref{mt6e3} and using \eqref{psi}, we obtain
\begin{equation}\label{mt6e4}
    \sum_{n=1}^{\infty}DSOME(n)q^n
=
(q^2;q^2)_\infty
\sum_{n=1}^{\infty}SOME(n)q^n
-
2\psi(q)
\sum_{n=1}^{\infty}SOME(n)q^{2n}.
\end{equation}
Employing \eqref{psi} and \eqref{f} in \eqref{mt6e4}, we obtain
$$
    \sum_{n=1}^{\infty}DSOME(n)q^n
=
\left(
\sum_{j=-\infty}^{\infty}(-1)^j q^{j(3j-1)}
\right)
\left(
\sum_{n=1}^{\infty}SOME(n)q^n
\right)$$\begin{equation}\label{mt6e5}
\hspace{3.4cm}-
\left(
\sum_{j=0}^{\infty}q^{j(j+1)/2}
\right)
\left(
\sum_{n=1}^{\infty}SOME(n)q^{2n}
\right).
\end{equation}
\eqref{mt6e5} is  equivalent to
\begin{align}
\sum_{n=1}^{\infty}DSOME(n)q^n
=&\sum_{n=1}^{\infty}\sum_{\substack{j\in\mathbb Z\\ j(3j-1)\leq n}}
(-1)^j
SOME\bigl(n-j(3j-1)q^n\notag\\
&\label{mt6e6}-
\sum_{n=1}^{\infty}\sum_{\substack{j\geq 0\\ j(j+1)/2\leq n\\ n\equiv j(j+1)/2\pmod 2}}
SOME\left(\dfrac{n-j(j+1)/2}{2}\right)q^{n}.
\end{align}
Comparing the coefficients of like powers of $q$ of \eqref{mt6e6}, we arrive at the desired result. 
\end{proof}

\begin{theorem}
For any integer $n\ge 1$, we have
\begin{align}
\overline{SOME}(n)
&=
SOME(n)+DSOME(n)
+
\sum_{\substack{r=1\\r\neq0}}^{\infty}(-1)^{r+1}
\overline{SOME}\left(n-\dfrac{r(3r-1)}{2}\right)\notag \\
&\label{MT7E1}\hspace{.5cm}+
\sum_{\substack{r=1\\r\neq0}}^{\infty}(-1)^r
SOME\left(n-r(3r-1)\right).\notag
\end{align}
\end{theorem}

\begin{proof}
We have
\begin{equation}\label{mt7e1}
2\sum_{m=1}^{\infty}
\dfrac{q^{2m-1}}{(1+q^{2m-1})^2}
=
\sum_{m=1}^{\infty}\frac{q^m}{(1+q^m)^2}
+
\sum_{m=1}^{\infty}
\dfrac{(-1)^{m-1}q^m}{(1+q^m)^2}.    
\end{equation}
Invoking \eqref{ie13} in \eqref{mt7e1} 
$$
\sum_{n=1}^{\infty}\overline{SOME}(n)q^n
=
\dfrac{(-q;q)_\infty}{(q;q)_\infty}
\left(
\sum_{m=1}^{\infty}\dfrac{q^m}{(1+q^m)^2}
+
\sum_{m=1}^{\infty}
\dfrac{(-1)^{m-1}q^m}{(1+q^m)^2}
\right).    
$$
Employing \eqref{i6} and \eqref{i7} in \eqref{mt7e3}, we obtain
\begin{equation}\label{mt7e3}
    \sum_{n=1}^{\infty}\overline{SOME}(n)q^n
=
(-q;q)_\infty\sum_{n=0}^{\infty}SOME(n)q^n
+
\dfrac{1}{(q;q)_\infty}\sum_{n=1}^{\infty}DSOME(n)q^n,
\end{equation}
which  is equivalent to
\begin{equation}\label{mt7e4}
    (q;q)_\infty\sum_{n=1}^{\infty}\overline{SOME}(n)q^n
=
(q^2;q^2)_\infty \sum_{n=0}^{\infty}SOME(n)q^n+\sum_{n=1}^{\infty}DSOME(n)q^n.
\end{equation}
Employing \eqref{f} in \eqref{mt7e4} and applying the Cauchy product of two infinite power series and then comparing the coefficients of the like powers of $q$, we arrive at the desired result.
\end{proof}

In the remaining theorems of this section, we will prove monocity results related to $SOME(n)$ and $\overline{SOME}(n)$ functions. We will using following truncated Pentagonal number theorem  introduced by Andrews and Merca  \cite{m_k}. For integers $n\ge 1$ and $k\ge 1$, 
\begin{equation}\label{mk}
   \dfrac{(-1)^{k-1}}{(q; q)_\infty} \sum_{n=1-k}^{k} (-1)^n q^{n(3n-1)/2} = (-1)^{k-1} + \sum_{n=0}^{\infty} M_k(n) q^n,
\end{equation} where  $M_k(n)$
is the number of partitions of $n$ in which $k$ is the least integer that is not a part and there are more parts $>k$ than there are $<k$.

\begin{theorem}\label{th1}
For any integer $n\geq2$  we have
$$
SOME(n)\geq SOME(n-2).    
$$
\end{theorem}
\begin{proof}
Replacing $q$ by $q^2$ in \eqref{mk}, we obtain
\begin{equation}\label{mt8e1}
\dfrac{(-1)^{k-1}}{(q^2;q^2)_\infty}
\sum_{j=1-k}^{k}(-1)^j q^{j(3j-1)}
=
(-1)^{k-1}
+
\sum_{\ell\geq 1}M_k(\ell)q^{2\ell}.    
\end{equation}
Multiplying both sides of \eqref{mt8e1} by $(q^2;q^2)_\infty
\sum_{n=0}^{\infty}SOME(n)q^n$, we obtain
\begin{align}
&(-1)^{k-1}
\left(
\sum_{j=1-k}^{k}(-1)^j q^{j(3j-1)}
\right)
\left(
\sum_{n=0}^{\infty}SOME(n)q^n
\right)\notag\\
&\label{mt8e2}=
(-1)^{k-1}(q^2;q^2)_\infty
\sum_{n=0}^{\infty}SOME(n)q^n
+
\left(\sum_{\ell=1}^{\infty}M_k(\ell)q^{2\ell}\right)
\left((q^2;q^2)_\infty
\sum_{n=0}^{\infty}SOME(n)q^n\right).
\end{align}
Employing \eqref{Tr} and \eqref{ie4} in \eqref{mt8e2}, we obtain
\begin{align}
   &(-1)^{k-1}
\left(
\sum_{j=1-k}^{k}(-1)^j q^{j(3j-1)}
\right)
\left(
\sum_{n=0}^{\infty}SOME(n)q^n
\right)\notag\\
&\label{mt8e3}=
(-1)^{k-1}\left(\sum_{m=0}^{\infty}p(m)q^{2m}\right)
\left(\sum_{r=1}^{\infty}T_rq^{T_r}\right)
+
\left(\sum_{\ell=1}^{\infty}M_k(\ell)q^{2\ell}\right)
\left(\sum_{m=0}^{\infty}p(m)q^{2m}\right)
\left(\sum_{r=1}^{\infty}T_rq^{T_r}\right). 
\end{align}
Now applying Cauchy product of infinite power series in \eqref{mt8e3} and then comparing the coefficients of the like powers of $q$,  we obtain
\begin{align}
(-1)^{k-1}
&\left(
\sum_{j=1-k}^{k}(-1)^j
SOME\bigl(n-j(3j-1)\bigr)
-
\sum_{\substack{r\geq 1\\ T_r\leq n\\ n-T_r\equiv 0\pmod 2}}
T_r\,p\!\left(\dfrac{n-T_r}{2}\right)
\right)\notag\\
&\label{mt8e4}=\sum_{\ell=1}^{\lfloor n/2\rfloor}M_k(\ell)\sum_{\substack{r\geq1\\ T_r\leq n-2\ell\\ n-2\ell-T_r\equiv0\pmod 2}}
T_r\,p\!\left(\dfrac{n-2\ell-T_r}{2}\right).
\end{align}
Putting  $k=1$ in \eqref{mt8e4} for $n\geq2$, we obtain
\begin{equation}\label{mt8e5}
    SOME(n)-SOME(n-2)\geq\sum_{\substack{r\geq 1\\ T_r\leq n\\ n-T_r\equiv 0\pmod 2}}
T_r\,p\!\left(\dfrac{n-T_r}{2}\right)\geq0.
\end{equation}
The desired result now follows from \eqref{mt8e5}. 
\end{proof}

 Following corollaries follow immediately from Theorem  \ref{th1}: 
\begin{corollary}  For any integer $n$, $SOME(n+1)\ngeq SOME(n)$.
	\end{corollary}
\begin{corollary}
    The sequences $\lbrace SOME(2n)\rbrace_{n\geq1}$ and $\lbrace SOME(2n-1)\rbrace_{n\geq1}$ are monotonic increasing sequences.
\end{corollary}
\begin{corollary}
   For  ordinary partitions of a positive integer $n$,  sum of all odd parts  in all partitions of $n$ is greater than or equal to the sum of all even parts.
\end{corollary}
\begin{proof}
    Using \eqref{mt8e5}, for $n\geq2$, we obtain
    \begin{align}
        &SOME(n)-SOME(n-2)\geq0,\notag\\
        &SOME(n-2)-SOME(n-4)\geq0,\notag\\
        &SOME(n-4)-SOME(n-6)\geq0,\notag\\
       &\hspace{4cm}\vdots\notag
         \end{align}
         Summing all of the above inequalities, we arrive at the desired result.
\end{proof}

\begin{theorem}\label{th2}
For any integer $n\geq2,$ we have
$$
\overline{SOME}(n)
\geq
\overline{SOME}(n-2).  
$$
\end{theorem}
\begin{proof} From \eqref{ie13} can be written as
Let 
\begin{equation}\label{mt9e1}
  \sum_{n= 0}^\infty\overline{SOME}(n)q^n= \dfrac{U(q)}{(q^2;q^2)_\infty},
\end{equation}
where 
$$
    U(q)
=
2(-q;q)_\infty^2
\sum_{m=1}^{\infty}
\dfrac{q^{2m-1}}{(1+q^{2m-1})^2} =2\sum_{m=1}^{\infty}
q^{2m-1}
\prod_{\substack{r=1\\ r\neq 2m-1}}^{\infty}
(1+q^r)^2,
$$
So for any integer  $n\geq0$, the coefficient of $q^n$ in $U(q),$ is 
$$
   U(n)\geq0,
$$
Multiplying \eqref{mt8e1}  by $U(q)$  and simplifying using \eqref{mt9e1}, we obtain
\begin{equation}\label{mt9e4}
(-1)^{k-1}
\left(
\sum_{j=1-k}^{k}
(-1)^j q^{j(3j-1)}
\right)
\left(
\sum_{n=0}^{\infty}\overline{SOME}(n)q^n
\right)
=
(-1)^{k-1}U(q)
+
\left(
\sum_{\ell=1}^{\infty}M_k(\ell)q^{2\ell}
\right)U(q). 
\end{equation}
Comparing the coefficient of $q^n$ on both sides of  \eqref{mt9e4}, we obtain
$$
(-1)^{k-1}
\sum_{j=1-k}^{k}
(-1)^j
\overline{SOME}\bigl(n-j(3j-1)\bigr)=(-1)^{k-1}U(n)+\sum_{\ell=1}^{\infty}M_k(\ell)U(n-2\ell).
$$
Therefore,
\begin{equation}\label{mt9e6}
(-1)^{k-1}
\left(
\sum_{j=1-k}^{k}
(-1)^j
\overline{SOME}\bigl(n-j(3j-1)\bigr)
-
U(n)
\right)
\geq 0.    
\end{equation}
Settying $k=1,$ in \eqref{mt9e6}, we obtain
\begin{equation}\label{mt9e7}
   \overline{SOME}(n)
-
\overline{SOME}(n-2)
-
U(n)
\geq 0.
\end{equation}
The desired result now easily follows from \eqref{mt9e7}.
\end{proof}
 Following corollaries are easy consequences of Theorem \ref{th2}. 
\begin{corollary} For any integer $n$, $\overline{SOME}(n+1)\ngeq \overline{SOME}(n)$.
\end{corollary}
\begin{corollary}
    The sequences $\lbrace \overline{SOME}(2n)\rbrace_{n\geq1}$ and $\lbrace \overline{SOME}(2n-1)\rbrace_{n\geq1}$ are monotonic increasing sequences.
\end{corollary}
\begin{corollary}Sum of all odd parts  is greater than or equal to the sum of all even parts for 
overpartitions of a positive integer $n$. 
\end{corollary}

\section{General analogue $S_{\mathcal P}(n)$ of $SOME(n)$ function }

Let $\lambda$ be any partition of a particular type $\mathcal{P}$ of a positive integer $n$. Define  the function $\omega(\lambda)$ by
$$
	\omega(\lambda)
	:=
	\sum_{\substack{\text{odd parts of }\lambda}} \text{part}
	-
	\sum_{\substack{\text{even parts of }\lambda}} \text{part}. $$
Let $\mathcal A(n)$ be the family of partitions of $n$ the type $\mathcal{P}$ such that $\lambda\in\mathcal A(n)$. Define the general analogue of $SOME(n)$, denoted by $S_{\mathcal P}(n)$,  as
\begin{equation}\label{b5}
	S_{\mathcal P}(n)
	=
	\sum_{\lambda\in \mathcal A(n)} \omega(\lambda).   
\end{equation}

\begin{theorem}Let $\lambda$ be any partition of a particular type $\mathcal{P}$ of a positive integer $n$. Then
	\begin{equation}\label{b3}
		\omega(\lambda)\equiv n\pmod 4.    
	\end{equation}
\end{theorem}
\begin{proof}
	For any  partition $\lambda$ of a particular type $\mathcal{P}$ of a positive integer $n$.  let
	$$
		O(\lambda):=\text{sum of odd parts of }\lambda
		\quad \mbox{and}\quad
		E(\lambda):=\text{sum of even parts of }\lambda.    
$$
	Then, clearly
	\begin{equation}\label{b1}
		O(\lambda)+E(\lambda)=n   
	\end{equation}
	and
	\begin{equation}\label{b2}
		O(\lambda)-E(\lambda)=\omega(\lambda).    
	\end{equation}
	Subtracting \eqref{b1} from \eqref{b2}, we obtain
	\begin{equation}\label{h1}
		\omega(\lambda)
		=
		n-2E(\lambda).    
	\end{equation}
	$E(\lambda)$  being a i sum of even parts,  $E(\lambda)$ is even. So, we can write
	$$
		E(\lambda)=2H(\lambda),    
	$$
	where
	\begin{equation}\label{h}
		H(\lambda)=\frac12 E(\lambda).    
	\end{equation}
	Therefore, employing \eqref{h} in  \eqref{h1}, we obtain
	\begin{equation}\label{h2}
		\omega(\lambda)=n-4H(\lambda).    
	\end{equation}
	The desired result now follows immediately from \eqref{h2}. 
\end{proof}

\begin{corollary}
	Let $\mathcal A(4n)$ be a family of partitions of $4n$ of a particular type $\mathcal P$. Then
	\begin{equation}\label{4n}
		S_{\mathcal P}(4n)\equiv 0 \pmod 4
	\end{equation}
	\begin{equation}\label{2n}
		S_{\mathcal P}(2n)\equiv 0 \pmod 2.
	\end{equation}
\end{corollary}
\begin{proof}
	For each partition $\lambda\in \mathcal A(4n)$, from \eqref{b3} it follows that 
	\begin{equation}\label{b4}
		\omega(\lambda)\equiv 4n\equiv 0 \pmod 4.    
	\end{equation}
	Emplyoing \eqref{b4} in \eqref{b5}. we arrive at \eqref{4n}.
	Sinmilarly,  \eqref{2n} can be proved.
\end{proof}

\begin{theorem}
	Let $\mathcal A(n)$ be any family of partitions of $n$ of partition type $\mathcal{P}$. For
	$r=0,1,\ldots,5$, let $T_r(n)$ denote the total number of parts congruent to
	$r \pmod 6$ occurring in all objects of $\mathcal A(n)$, counted with multiplicity. Then the analogue of $SOME(n),$ $S_{\mathcal P}(n)$  on $\mathcal{P}$ satisfies
	$
	S_{\mathcal P}(n)
	\equiv
	0
	\pmod 3$ if and only if $T_1(n)+T_2(n)\equiv T_4(n)+T_5(n)\pmod 3.
	$
\end{theorem}
\begin{proof}
	A part $a\in\lambda\in\mathcal A(n)$ contributes $+a$ to $
	S_{\mathcal P}(n)$ if $a$ is odd and $-a$ if
	$a$ is even. So by reducing modulo $3$ and classifying the parts modulo $6$,
	we obtain
	$$
	\begin{array}{c|c|c}
		a\pmod 6 & \text{$S_{\mathcal P}(n)$ contribution} & \text{contribution modulo }3\\
		\hline
		0 & -a & 0\\
		1 & +a & 1\\
		2 & -a & 1\\
		3 & +a & 0\\
		4 & -a & -1\\
		5 & +a & -1\\
		\hline
	\end{array}
	$$
	Thus, parts congruent to $1$ or $2 \pmod 6$ contribute $+1$ modulo $3$; parts
	congruent to $4$ or $5 \pmod 6$ contribute $-1$ modulo $3$, and parts congruent
	to $0$ or $3 \pmod 6$ contribute $0$ modulo $3$. Summing up  these contributions over all partitions in
	$\mathcal A(n)$, we obtain
	\begin{equation}\label{Tb2}
		S_{\mathcal P}(n)
		\equiv
		T_1(n)+T_2(n)-T_4(n)-T_5(n)
		\pmod 3.    
	\end{equation}
	Now, the required result  follows immediately from \eqref{Tb2}.
\end{proof}

\begin{theorem}
	Let $k$ be a positive integer and $L=\operatorname{lcm}(2,k).$ For $r=0,1,\ldots,L-1$, let $M_r(n)$ denote the total number of parts congruent
	to $r\pmod L$ occurring in all objects of $\mathcal A(n)$ of partition type $\mathcal{P}$, counted with
	multiplicity. Then
	$
	S_{\mathcal P}(n)
	\equiv
	0
	\pmod k$
	if and only if $
	\sum_{r=0}^{L-1}
	(-1)^{r+1}r\,M_r(n)
	\equiv 0\pmod k.
	$
\end{theorem}
\begin{proof}
	If any part $a$ of the partition $\lambda\in\mathcal A(n)$ satisfies $a\equiv r\pmod L$, then $a\equiv r\pmod k$. Therefore, the contribution of $a$ modulo $k$ to $S_{\mathcal P}(n)$  is $r \pmod k$ if $r$ is odd and  $-r \pmod k$ if $r$ is even.
	Summing up  these contributions over all parts in all partitions of  $\mathcal A(n)$, we obatin
	\begin{equation}\label{Tb3}
		S_{\mathcal P}(n)
		\equiv
		\sum_{r=0}^{L-1}
		(-1)^{r+1}r\,M_r(n)
		\pmod k.    
	\end{equation}
	Now the desire result immediately follows from \eqref{Tb3}.
\end{proof}

\section{Colour partition analogue $S_c(n)$ of $SOME(n)$ function}

Let  $\mathcal P_{\mathbf c}$  denote the colour partition of a positive integer $n$ of the type  in which  a part of size $j$ in any partition of $n$ appears in $c_j$ different colors, where $c_j\ge 0$.
 If $A_c(n)$ denotes the total number partitions of a positive integer $n$ of the type $\mathcal P_c$, then its generating function is given by 
$$
\Gamma(q)
:=
\sum_{n=0}^{\infty}A_{\mathbf c}(n)q^n
=
\prod_{j=1}^{\infty}\dfrac{1}{(1-q^j)^{c_j}}.
$$
Let $S_{\mathbf c}(n)$ denote the corresponding $SOME(n)$ function defined over all the partitions of $n$ of the type $\mathcal P_c$.
\begin{theorem}
	We have 
	\begin{equation}\label{Tc1}
		\sum_{n=0}^\infty S_{\mathbf c}(n)q^n
		=
		\Gamma(q)
		\sum_{j=1}^{\infty}
		(-1)^{j+1}j c_j
		\frac{q^j}{(1-q^j)}.
	\end{equation}
\end{theorem}
\begin{proof}
	Let, 
	$$
		\Gamma(z,q)
	=
	\prod_{\substack{j=1\\ j\text{ odd}}}^{\infty}
	(1-z^j q^j)^{-c_j}
	\prod_{\substack{j=1\\ j\text{ even}}}^{\infty}
	(1-z^{-j}q^j)^{-c_j}.
	$$
	Then,
	\begin{equation}\label{ptc1}
		\sum_{n=0}^{\infty}S_{\mathbf c}(n)q^n
		=
		\left.
		\dfrac{\partial}{\partial z}	\Gamma(z,q)
		\right|_{z=1}.
	\end{equation}
	By logarithmic differentiation of \eqref{ptc1}, we  obtain
	$$
	\left.
	\dfrac{\partial}{\partial z}\Gamma(z,q)
	\right|_{z=1}
	=
	\Gamma(q)
	\sum_{j=1}^{\infty}
	(-1)^{j+1}j c_j
	\frac{q^j}{(1-q^j)}.
	$$
	Hence, the proof is complete.
\end{proof}
In the following theorem, we will use well known Möbius inversion  formula for arithmetic functions. For any two  functions $F$ and $f$ defined over the set of positive integers, 
$$
	F(n)=\sum_{d\mid n} f(d),    
$$
if and only if
$$
	f(n)=\sum_{d\mid n}\mu(d)\,F\!\left(\frac{n}{d}\right),    
$$
where $\mu$ denotes the Möbius function. 
\begin{theorem}\label{t5}
	For every integer  $n\geq1$,  $S_{\mathbf c}(n)\equiv 0\pmod k$ if and only if $k|jc_j$ for all $j\geq1.$
\end{theorem}
\begin{proof}
	From \eqref{Tc1}, we have
	$$
	\sum_{n=0}^{\infty}S_{\mathbf c}(n)q^n\equiv 0\pmod k
	$$
	if and only if
	$$\Gamma(q)
	\sum_{j=1}^{\infty}
	(-1)^{j+1}j c_j
	\frac{q^j}{(1-q^j)}
	\equiv 0\pmod k.
	$$
Now, let$$\Gamma(q)=1+\gamma_1q+\gamma_2q^2+\gamma_3q^3+\cdots, $$
and $$
L(q)=\sum_{j=1}^{\infty}
	(-1)^{j+1}j c_j
	\frac{q^j}{(1-q^j)}=l_1q+l_2q^2+l_3q^3+\cdots .$$
Then
\begin{equation}\label{GL}
\Gamma(q)L(q)
=
\left(1+\gamma_1q+\gamma_2q^2+\gamma_3q^3+\cdots\right)
\left(l_1q+l_2q^2+l_3q^3+\cdots\right).    
\end{equation}\\
Note that from \eqref{GL}, the coefficient of $q$ in $\Gamma(q)L(q)$ is $l_1$. Hence, if
$$
\Gamma(q)L(q)\equiv0 \pmod{k},
$$
then
$$
l_1\equiv0 \pmod{k}.
$$
Also from \eqref{GL} the coefficient of $q^2$ in $\Gamma(q)L(q)$ is
$$
l_2+\gamma_1l_1.
$$
Since $l_1\equiv0 \pmod{k}$, it follows that
$$
l_2\equiv0 \pmod{k}.
$$
Similarly, the coefficient of $q^3$ in $\Gamma(q)L(q)$ is
$$
l_3+\gamma_1l_2+\gamma_2l_1.
$$
Since $l_1\equiv l_2\equiv0 \pmod{k}$, we obtain
$$
l_3\equiv0 \pmod{k}.
$$
Continuing in this way,  by induction we obtain
$$
l_n\equiv0 \pmod{k}
\quad\text{for every } n\geq1.
$$
Therefore,
$$
\Gamma(q)L(q)\equiv0 \pmod{k}
\quad\text{if and only if}\quad
L(q)\equiv0 \pmod{k}.
$$
Also, since
	$$
	\dfrac{q^j}{1-q^j}
	=
	q^j+q^{2j}+q^{3j}+\cdots,
	$$
	the coefficient of $q^n$ in $L(q)=\sum_{j=1}^{\infty}
	(-1)^{j+1}j c_j
	\dfrac{q^j}{1-q^j}
	$ is $\sum_{j\mid n}(-1)^{j+1}j c_j.$ Therefore,  $$\sum_{j=1}^{\infty}
	(-1)^{j+1}j c_j
	\dfrac{q^j}{1-q^j}
	\equiv 0\pmod k \text{ if and only if }\sum_{j\mid n}(-1)^{j+1}j c_j
	\equiv 0\pmod k,$$
    for every $n\geq1.$
	Set, $$
	a_j=(-1)^{j+1}j c_j.
	$$
	Then, by Mobius inversion formula, for any integer $n\geq1$,
	$$
	\sum_{j\mid n}a_j\equiv 0\pmod k 
	$$
	 if and only if $a_j\equiv 0\pmod k$ 
	for all  $j\geq1$. Thus,
	$
	(-1)^{j+1}j c_j\equiv 0\pmod k \text{ if and only if }k\mid j c_j
	$
	Hence,  the proof is complete. \end{proof}
Following  corollaries follow easily from Theorem  \ref{t5}. 

\begin{corollary} For any integer $n\ge 1$, 
	$
	S_{\mathbf c}(n)\equiv0\pmod3
	$
	if and only if
	$
	c_j\equiv0\pmod3$
	whenever $3\nmid j
	$ and 
	parts divisible by $3$ can appear in any colour. 
\end{corollary}

\begin{corollary}
	For any integer $n\ge 1$, 
	$
	S_{\mathbf c}(n)\equiv0\pmod4
	$
	if and only if
	$$
	\begin{cases}
		4\mid c_j, & j\equiv1,3\pmod4,\\
		2\mid c_j, & j\equiv2\pmod4,\\
		\text{no restriction on }c_j, & j\equiv0\pmod4.
	\end{cases}
	$$
\end{corollary}

\begin{corollary}
	Let $\mathcal{P}_{ck}$ denote the colour partition type of a positive integer  $n$ in which  each part $j$ is a multiple of $k$ and appears in $c_j$ different colours. 
	Let   $S_{ck}(n)$ be the corresponding $SOME(n)$ function defined over all partitions of $n$ of type  $\mathcal{P}_{ck}$. Then, for any integer $n\ge 1$, 
$$
	S_{ ck}(n)\equiv0\pmod k.
$$
\end{corollary}

\begin{corollary} For any positive integer $n$, the $SOME(n)$ function defined over all the partitions of a positive integer  $n$ such that each part  appearing in $k$ different colours in any partition of $n$  is divisible by $k$. 
\end{corollary}

   \section*{\bf Declarations}
   
   \noindent{\bf Author Contributions.} Both authors contributed equally to this work.
   
 \noindent{\bf Conflict of Interest.} The authors declare that there is no conflict of interest regarding the publication of this article.
 
 \noindent{\bf Human and Animal Rights.} The authors declare that there is no research involving human participants or animals in the context of this paper.	
 
 \noindent{\bf Data Availability Statements.} Data sharing not applicable to this article as no datasets were generated or analyzed during the current study.	
\bibliographystyle{plain}

\end{document}